\documentclass[reqno]{amsart}

\usepackage{amssymb}
\usepackage{graphicx}
\usepackage[usenames, dvipsnames]{color}
\usepackage{verbatim}
\usepackage{mathrsfs}
\usepackage{esint}
\usepackage{empheq}
\usepackage[colorlinks,linkcolor=blue,citecolor=red]{hyperref}

\numberwithin{equation}{section}

\newtheorem{theorem}{Theorem}[section]

\newtheorem{lemma}[theorem]{Lemma}

\theoremstyle{definition}
\newtheorem{remark}[theorem]{Remark}
\theoremstyle{definition}

\theoremstyle{definition}

\makeatletter
\def\dashint{\operatorname%
{\,\,\text{\bf-}\kern-.98em\DOTSI\intop\ilimits@\!\!}}
\makeatother

\begin{document}
\title[Parabolic $k$-Hessian equations]
{Interior estimates of derivatives and a Liouville type theorem for parabolic $k$-Hessian equations}

\author[J. Bao]{Jiguang Bao}
\address[J. Bao]{School of Mathematical Sciences, Beijing Normal University, Laboratory of Mathematics and Complex Systems, Ministry of Education, Beijing 100875, China.}
\email{jgbao@bnu.edu.cn}
\thanks{J. Bao was supported by NSFC (No. 11871102)}

\author[J. Qiang]{Jiechen Qiang}
\address[J. Qiang]{School of Mathematical Sciences, Beijing Normal University, Laboratory of Mathematics and Complex Systems, Ministry of Education, Beijing 100875, China.}
\email{qiangjc@mail.bnu.edu.cn}

\author[Z. Tang]{Zhongwei Tang}
\address[Z. Tang]{School of Mathematical Sciences, Beijing Normal University, Laboratory of Mathematics and Complex Systems, Ministry of Education, Beijing 100875, China.}
\email{tangzw@bnu.edu.cn}
\thanks{Z. Tang was supported by NSFC (Nos. 12071036 and 12126306)}

\author[C. Wang]{Cong Wang}
\address[C. Wang]{School of Mathematical Sciences, Beijing Normal University, Laboratory of Mathematics and Complex Systems, Ministry of Education, Beijing 100875, China.}
\email{cwang@mail.bnu.edu.cn}

%\date{\today}



\begin{abstract}
 In this paper, we establish the gradient and Pogorelov estimates for $k$-convex-monotone solutions to parabolic $k$-Hessian equations of the form $-u_t\sigma_k(\lambda(D^2u))=\psi(x,t,u)$. We also apply such estimates to obtain a Liouville type result, which states that any $k$-convex-monotone and $C^{4,2}$ solution $u$ to $-u_t\sigma_k(\lambda(D^2u))=1$ in $\mathbb{R}^n\times(-\infty,0]$ must be a linear function of $t$ plus a quadratic polynomial of $x$, under some growth assumptions on $u$.
\end{abstract}
\maketitle

\section{Introduction}

This paper is concerned with the interior estimates of gradient and second derivatives for solutions to parabolic $k$-Hessian equations of the form
\begin{equation}\label{eq:P-k-HESSIAN}
  -u_t\sigma_k(\lambda(D^2u))=\psi(x,t,u)
\end{equation}
in a bounded domain $D\subset\mathbb{R}^n\times(-\infty,0]$, where $D^2u$ is the matrix of second derivatives of $u=u(x,t)$ with respect to $x$, $\lambda(D^2u)$ is the eigenvalue vector $\lambda=(\lambda_1,\cdots,\lambda_n)$ of $D^2u$,
$$\sigma_k(\lambda)=\sum_{1\leq i_1<\cdots< i_k\leq n}\lambda_{i_1}\cdots\lambda_{i_k}$$
is the $k$-th elementary symmetric function, $k=1,\cdots,n$, and $\psi$ is a nonnegative function. For $2\leq k\leq n$, \eqref{eq:P-k-HESSIAN} is an important class of fully nonlinear parabolic equations. Especially, for $k=n$, the left hand of \eqref{eq:P-k-HESSIAN} corresponds to the parabolic Monge-Amp\`{e}re operator of the form $-u_t\det D^2u$, which was first introduced by Krylov \cite{Krylov-1976} and is closely associated with the Gauss-Kronecker curvature flows \cite{Tso-1985}.

A priori estimates for the elliptic $k$-Hessian equations
\begin{equation}\label{eq:k-Hessian}
  \sigma_k(\lambda(D^2u))=\psi(x,u),
\end{equation}
in a bounded domain $\Omega\subset\mathbb{R}^n$, have been researched extensively. For $k=1$, \eqref{eq:k-Hessian} corresponds to $\Delta u=\psi(x,u)$, a linear or quasilinear equation. Its priori estimates have sufficiently developed (see \cite{GT-1983}). For $2\leq k\leq n$, \eqref{eq:k-Hessian} is a fully nonlinear equation. As is well known, its natural ellipticity class is $k$-convex functions. For $k=n$, \eqref{eq:k-Hessian} is the Monge-Amp\`{e}re equation. Interior gradient estimate for Monge-Amp\`{e}re equations is immediate due to the convexity of solutions. For Monge-Amp\`{e}re equations with the homogeneous Dirichlet boundary condition on $\partial\Omega$, Pogorelov \cite{Pogorelov-1978} first introduced Pogorelov estimates. In the context of fully nonlinear elliptic equations, Pogorelov estimates is a significant type of interior second derivatives estimates.
See also \cite{Blocki-2003, Figalli-2017, Gutierrez-2001, JT-2014, LT-2010, Maldonado-2018} for various version of Pogorelov estimates for Monge-Amp\`{e}re equations.   By contrast, it is much harder to handle the intermediate cases $2\leq k\leq n-1$. In \cite{Chou-Wang 2003} and \cite{Trudinger-1997}, the authors proved the interior gradient estimates for $k$-convex solutions. Chou-Wang \cite{Chou-Wang 2003} extended Pogorelov estimates \cite{Pogorelov-1978} to $k$-Hessian equations for $k$-convex solutions, with some general Dirichlet boundary conditions. Afterwards, Warren-Yuan \cite{Warren-Yuan-2009} established interior $C^2$ estimates for $\sigma_2(\lambda(D^2u))=1$ in space dimension $3$. Recently, McGonagle-Song-Yuan \cite{MSY-2019} further derived the interior $C^2$ estimates for semiconvex solutions to $\sigma_2(\lambda(D^2u))=1$. When $\psi$ in \eqref{eq:k-Hessian} depends additionally on the gradient term $Du$, Li-Ren-Wang \cite{LRW2016} obtained Pogorelov estimates for $(k+1)$-convex solutions.

In this paper, we will first establish the interior gradient and Pogorelov estimates for $k$-convex-monotone solutions to \eqref{eq:P-k-HESSIAN}. This extends the previously corresponding results on elliptic $k$-Hessian equations \cite{Chou-Wang 2003}.

To state our results precisely, we recall the following definitions related to parabolic $k$-Hessian equations (see \cite{Lieberman-1996}).

Let $D$ be a bounded open set in $\mathbb{R}^n\times(-\infty,0]$. We define for $t\leq0$, $D(t)=\{x\in\mathbb{R}^n:(x,t)\in D\}$
and $\underline{t}=\inf\{t\leq0:D(t)\neq\emptyset\}$. The parabolic boundary $\partial_pD$ of $D$ is defined by
$$\partial_pD=\left(\overline{D}(\underline{t})\times\{\underline{t}\}\right)\cup\bigcup_{t\leq0}(\partial D(t)\times\{t\}),$$
where $\overline{D}(\underline{t})$ denotes the closure of $D(\underline{t})$, $\partial D(t)$ denotes the boundary of $D(t)$.
For an open set $\Omega$ in $\mathbb{R}^n$, we say that a function $u=u(x)\in C^2(\Omega)$ is $k$-convex, if $\lambda(D^2u(x))\in\overline{\Gamma}_k$ for every $x\in\Omega$, where $\Gamma_k$ is an open convex symmetric cone in $\mathbb{R}^n$ with its vertex at the origin, given by
$$\Gamma_k=\{\lambda\in\mathbb{R}^n: \sigma_j(\lambda)>0,\ \forall\ j=1,\cdots,k\}.$$
We say that a function $u=u(x,t)$ is $k$-convex-monotone, if $u$ is $k$-convex in $x$ and decreasing in $t$.
%Let $\psi$ be a function defined in $\overline{\Gamma}=\overline{D}\times\mathbb{R}$.
We denote points in $\overline{D}\times\mathbb{R}$ typically by $(x,t,z)$, where $(x,t)\in\overline{D}$ and $z\in\mathbb{R}$.
%For $\psi$ defined in $\overline{D}\times\mathbb{R}$, we write
%$$|\psi_x|=|(\psi_{x_1},\cdots,\psi_{x_n})|,\quad|\psi_{xx}|=|(\psi_{x_ix_j})|.$$
%We write
%$$||\psi||_{C^{0+1}(\overline{\Gamma})}=\sup_{\Gamma}|\psi|+\sup_{(x,t,z_1),(y,s,z_2)\in\Gamma\atop(x,t,z_1)\neq(y,s,z_2)}
%\frac{|\psi(x,t,z_1)-\psi(y,s,z_2)|}{|(x,t,z_1)-(y,s,z_2)|},$$
%$$||\psi||_{C^{1+1}(\overline{\Gamma})}=\sup_{\Gamma}|\psi|+\sum_{i=1}^n||\psi_{x_i}||_{C^{0+1}(\overline{\Gamma})}
%+||\psi_t||_{C^{0+1}(\overline{\Gamma})}+||\psi_{z}||_{C^{0+1}(\overline{\Gamma})}.$$

The main results of the first part now can be stated as below.

\begin{theorem}\label{thm:gradient-est}
  Let $D=\big\{(x,t)\in\mathbb{R}^n\times(-\infty,0]:|x|^2-r^2<t\big\}$ with $r>0$ being a constant. Let $\psi\geq0$ be Lipschitz continuous in $\overline{D}\times\mathbb{R}$ with respect to $x$ and $z$, respectively. Assume that $u\in C^{4,2}(D)\cap C^{2,1}(\overline{D})$ is a $k$-convex-monotone solution to \eqref{eq:P-k-HESSIAN} satisfying
  \begin{equation}\label{eq:u-t-bdd0}
  -u_t\geq m_1>0\quad\text{in }D,
  \end{equation}
  for some constant $m_1$. Then
  $$|Du(0,0)|\leq C_1+C_2\frac{\sup_D|u|}{r},$$
  where $C_1>0$ is a constant depending only on $n$, $k$, $r$, $m_1$, $\sup_D|u|$, $\sup_{D\times\mathbb{R}}(|\psi_{x}|+|\psi_z|)$, and $C_2>0$ is a constant depending only on $n$, $k$, $m_1$, $\sup_{D\times\mathbb{R}}\psi$. In particular, if $\psi$ is a constant, then $C_1=0$.
\end{theorem}

\begin{theorem}\label{thm:Pogorelov}
  Let $D$ be a bounded domain in $\mathbb{R}^n\times(-\infty,0]$. Let $\psi\geq\psi_0>0$ in $\overline{D}\times\mathbb{R}$ for some constant $\psi_0$, which satisfies $\psi_x(\psi_z)$ is Lipschitz continuous with respect to $x(z)$. Let $u\in C^{4,2}(D)\cap C^{2,1}(\overline{D})$ be a $k$-convex-monotone solution to \eqref{eq:P-k-HESSIAN}
  satisfying \eqref{eq:u-t-bdd0}. Assume that there is a $k$-convex-monotone function $w\in C^{2,1}(\overline{D})$ such that
  $$w\geq u\ \text{in }D\quad\text{and}\quad w=u\ \text{on } \partial_pD.$$
  Then
  $$(w-u)^4|D^2u|\leq C\quad\text{in }D,$$
  where $C>0$ is a constant depending only on $n$, $k$, $m_1$, $\sup_D|w-u|,\sup_D|D(w-u)|$, $\psi_0$ and $\sup_{D\times\mathbb{R}}(\psi+|\psi_x|+|\psi_z|+|\psi_{xx}|+|\psi_{zz}|)$.
\end{theorem}

\begin{remark}
  In Theorem \ref{thm:Pogorelov}, if $u$ takes constant boundary data $u_0$ on $\partial_pD$, then we can take $w\equiv u_0$ in $\overline{D}$ by the strong maximum principle for parabolic $k$-Hessian equations (see for instance \cite{Lieberman-1996}).
\end{remark}

For Theorem \ref{thm:gradient-est} and Theorem \ref{thm:Pogorelov}, the basic idea of the proofs is to construct proper auxiliary functions and take advantage of the associated linearized operator, which is analogous to those presented in the related works, for instance, \cite{Blocki-2003,Chou-Wang 2003,Guan-1997,Trudinger-1997}. Such interior estimates are delicate since the auxiliary function must be chosen carefully, see \cite[Chapter 17]{GT-1983} for detailed explanations. Based on the work of \cite{Chou-Wang 2003}, we consider the auxiliary function involving the term $t$ and finally derive the estimates.

The interior estimates of derivatives play a dominant role in investigating the regularity and Liouville property of solutions, which are fundamental problems for associated equations. The second part of this paper is concerned with the Liouville property of solutions to parabolic $k$-Hessian equations
\begin{equation}\label{eq:psigma_k=1}
   -u_t\sigma_k(\lambda(D^2u))=1\quad\text{in }\mathbb{R}^n\times(-\infty,0].
\end{equation}

We review some known results related to the elliptic $k$-Hessian equations
\begin{equation}\label{eq:sigma_k=1}
\sigma_k(\lambda(D^2u))=1\quad\text{in }\mathbb{R}^n.
\end{equation}
For $k=1$, the classical Liouville theorem for harmonic functions illustrates that any convex classical solutions to \eqref{eq:sigma_k=1} must be quadratic. For $k=n$, the celebrated theorem to J\"{o}rgens \cite{Jorgens1954}, Calabi \cite{Calabi1958} and Pogorelov \cite{P 1972} for the Monge-Amp\`{e}re equation states that any convex classical solution to \eqref{eq:sigma_k=1} must be quadratic.
Bao-Chen-Guan-Ji \cite{Bao2003} proved that any strictly convex solution to \eqref{eq:sigma_k=1} with a lower quadratic growth condition is quadratic.
Very recently, Du \cite{Du-2022} proved three necessary and sufficient conditions of Liouville property for convex solutions to \eqref{eq:sigma_k=1}.
For $k=2$, Chang-Yuan \cite{CY2010} demonstrated that if $u$ is a solution to \eqref{eq:sigma_k=1} with $D^2u\geq\left(\delta-\sqrt{\frac{2}{n(n-1)}}\right)I$
for some $\delta>0$, then $u$ is quadratic. Recently, Shankar-Yuan \cite{SY2021} improved this result to general semiconvex solutions, namely $D^2u\geq -KI$
for a large $K>0$. Li-Ren-Wang \cite{LRW2016} reduced the strict convexity restriction in \cite{Bao2003} to $(k+1)$-convexity.
In the above works \cite{Bao2003,CY2010,Du-2022,LRW2016,SY2021}, the authors dealt with solutions satisfying the restricted $(k+1)$-convexity, $n$-convexity or its variations.
However, as far as we know, little is known about the Liouville type results for general $k$-convex solutions to \eqref{eq:sigma_k=1}. We refer to \cite{B 2004, Wang-Bao 2022}. Wang-Bao \cite{Wang-Bao 2022} obtained such a result under an additional upper quadratic growth condition lately.

Contrast to the elliptic $k$-Hessian case, there are less known Liouville type results for the parabolic $k$-Hessian equations, especially for $k$-convex-monotone solutions. Guti\'{e}rrez-Huang \cite{GH-1998} extended the J\"{o}rgens-Calabi-Pogorelov theorem to the parabolic setting
\begin{equation}\label{eq:PMA}
  -u_t\det D^2u=1.
\end{equation}
In \cite{Wang-Bao-2015} and \cite{Xiong-Bao-2011}, the authors showed the Liouville type theorem for another two kinds of parabolic Monge-Amp\`{e}re equations
$$-u_t=(\det D^2u)^{\frac{1}{n}}$$
and
$$-u_t+\log\det(D^2u)=1,$$
respectively. For general $k$, Nakamori-Takimoto \cite{Nakamori 2015} proved the Liouville type theorem for convex-monotone solutions to \eqref{eq:psigma_k=1} with a lower quadratic growth condition. More recently, He-Sheng-Xiang-Zhang \cite{HSXZ-2022} improved this result to $(k+1)$-convex-monotone solutions. In the current work, we make an attempt to obtain a Liouville type result for $k$-convex-monotone solutions to \eqref{eq:psigma_k=1}, when the solution further satisfies an upper quadratic growth condition.

The main result of our second part is the following.

\begin{theorem}\label{thm:main}
  Let $u\in C^{4,2}(\mathbb{R}^n\times(-\infty,0])$ be a $k$-convex-monotone solution to \eqref{eq:psigma_k=1}. Suppose that there exist positive constants $m_2\geq m_1$ such that
  \begin{equation}\label{eq:u-t-bdd}
     m_1\leq-u_t\leq m_2\quad\text{in }\mathbb{R}^n\times(-\infty,0],
  \end{equation}
  and that there exist positive constants $A_1$, $A_2$ and $B$ such that
  \begin{equation}\label{eq:quadratic}
  A_1|x|^2\leq u(x,0)\leq A_2|x|^2+B,\quad\forall\ x\in\mathbb{R}^n.
  \end{equation}
  Then $u$ has the form $u(x,t)=-mt+p(x)$, where $m>0$ is a constant and $p(x)$ is a quadratic polynomial.
\end{theorem}

\begin{remark}
  When $k=n$, \eqref{eq:psigma_k=1} corresponds to the parabolic Monge-Amp\`{e}re equation \eqref{eq:PMA}. By Theorem 1.1 in \cite{GH-1998}, the assumption \eqref{eq:quadratic} can be removed.
\end{remark}

We note that it is expected and intrinsic to treat the interior estimates of derivatives and Liouville property of $k$-convex-monotone solutions, due to the natural parabolic-class for parabolic $k$-Hessian equations.
%The advantage of the upper growth condition lie in the fact that it provides an of sublevel set.

The paper is organized as follows. In Section \ref{Sec2} and Section \ref{Sec3}, we derive the interior gradient estimates and Pogorelov estimates for parabolic $k$-convex solutions to \eqref{eq:P-k-HESSIAN}, respectively. In Section \ref{Sec4}, we give the proof of Liouville type theorem (Theorem \ref{thm:main}).

We fix some notations throughout this paper. We denote $B_r$ as the ball in $\mathbb{R}^n$ centered at $0$ of radius $r$.
For $u$ defined in a domain $D$ in $\mathbb{R}^n\times(-\infty,0]$, we denote $D^2u(x,t)=(D_{ij}u(x,t))$ the matrix of second derivatives of $u$ with respect to $x$, $Du(x,t)=(D_1u(x,t),\cdots,D_nu(x,t))$ the gradient of $u$ with respect to $x$, $u_t(x, t)$ the derivative of $u$ with respect to $t$. For simplicity, we further denote $D_iu$ by $u_i$, $D_{ij}u$ by $u_{ij}$ and so on. We write $S_k=S_k(D^2u)=\sigma_k(\lambda(D^2u))$, and use the Einstein summation convention when there is no ambiguity.

%We also denote $C(m_1,\cdots,m_j)$ as some positive constant depending only on $m_1,\cdots,m_j$ that may vary from line to line.

\section{Interior gradient estimates for $k$-convex-monotone solutions}\label{Sec2}

In this section, we derive a specific form of an interior gradient estimate for $k$-convex-monotone solutions to equation \eqref{eq:P-k-HESSIAN}, which will be used to prove the Liouville type theorem in the last section. The estimate is also of independent interest. Our derivation is similar in spirit to that in \cite{Chou-Wang 2003} where the $k$-Hessian equations are considered.

\begin{proof}[Proof of Theorem \ref{thm:gradient-est}]
  We assume $u\not\equiv0$ in $D$. We consider the auxiliary function, for $(x,t)\in\overline{D}$ and $\xi\in\partial B_1$,
  $$G(x,t;\xi)=\rho(x,t)\varphi(u(x,t))u_{\xi}(x,t),$$
  where $\rho(x,t)=1-\frac{|x|^2-t}{r^2}$, $\varphi(s)=(M-s)^{-\frac{1}{2}}$ and $M=4\sup_D|u|$. Assume $G$ attains its maximum at $(x_0,t_0;\xi_0)$. It suffices to prove
  \begin{equation}\label{eq:G}
  \rho(x_0,t_0)u_{\xi_0}(x_0,t_0)\leq C_1+C_2\frac{M}{r}.
  \end{equation}
  We may assume $(x_0,t_0)\in\overline{D}\setminus\partial_pD$, as otherwise \eqref{eq:G} becomes trivial. We may also assume $\xi_0=(1,0,\cdots,0)$.
  Clearly, $u_i(x_0,t_0)=0$ for $i=2,\cdots,n$.

  We next employ the function
  $$H(x,t)=\rho(x,t)\varphi(u(x,t))u_1(x,t)\quad\text{for }(x,t)\in\overline{D}.$$
  It follows that, at $(x_0,t_0)$,
  \begin{equation}\label{eq:DH}
    0=H_i=\rho\varphi u_{1i}+\rho\varphi'u_1u_i+\rho_i\varphi u_1,
  \end{equation}
  \begin{equation}\label{eq:H_t}
    0\leq H_t=\rho\varphi u_{1t}+\rho\varphi'u_1u_t+\rho_t\varphi u_1,
  \end{equation}
  and the Hessian matrix
  \begin{equation}\label{eq:D^2H}
  \begin{split}
    \{H_{ij}\}= &\ \big\{\rho\varphi u_{1ij}+\rho\varphi'u_1u_{ij}+\rho\varphi''u_1u_iu_j+\rho_{ij}\varphi u_1 \\
    & +\rho\varphi'(u_{1i}u_j+u_{1j}u_i)+\varphi(\rho_ju_{1i}+\rho_iu_{1j})+\varphi'u_1(\rho_ju_i+\rho_iu_j)\big\}
  \end{split}
  \end{equation}
  is nonpositive definite, for $i,j=1,2,\cdots,n$. Let $L$ be the linearized parabolic $k$-Hessian operator associated with $u$ at $(x_0,t_0)$ defined by
  $$L=\frac{\psi(x_0,t_0,u(x_0,t_0))}{u_t(x_0,t_0)}\frac{\partial}{\partial t}-u_t(x_0,t_0)S_k^{ij}(D^2u(x_0,t_0))D_{ij},$$
  where
  $$S_k^{ij}=S_k^{ij}(D^2u)=\frac{\partial S_k}{\partial u_{ij}}(D^2u).$$
  In the rest of the proof, all the calculations are made at $(x_0,t_0)$. By \eqref{eq:H_t} and \eqref{eq:D^2H}, direct calculation gives
  \begin{equation}\label{eq:LH}
    \begin{split}
       0\geq &\ LH \\
         = &\ \rho\varphi\psi\frac{u_{1t}}{u_t}+\rho\varphi'\psi u_1+\varphi\psi u_1\frac{\rho_t}{u_t}-u_t\big(\rho\varphi S_k^{ij}u_{1ij}+\rho\varphi'u_1S_k^{ij}u_{ij}+\rho\varphi''u_1S_k^{ij}u_iu_j \\
          &\ +\varphi u_1S_k^{ij}\rho_{ij}+2\varphi'u_1S_k^{ij}\rho_iu_j+2S_k^{ij}u_{1i}(\rho\varphi'u_j+\varphi\rho_j)\big).
    \end{split}
  \end{equation}
  Differentiating \eqref{eq:P-k-HESSIAN} with respect to $x_1$ yields
  \begin{equation}\label{eq:H0}
   -u_{1t}S_k-u_tS_k^{ij}u_{1ij}=\partial_1\psi:=\psi_{x_1}+\psi_zu_1.
  \end{equation}
  By \eqref{eq:DH}, we have
  \begin{equation}\label{eq:H1}
    u_{1i}=-\frac{u_1}{\rho\varphi}(\rho\varphi'u_i+\rho_i\varphi)\quad\text{for }i=1,2,\cdots,n.
  \end{equation}
  Since $S_k$ is homogeneous of degree $k$, we have
  \begin{equation}\label{eq:k-homo}
    S_k^{ij}u_{ij}=kS_k=-\frac{k\psi}{u_t}.
  \end{equation}
  Substituting \eqref{eq:H0}-\eqref{eq:k-homo} into \eqref{eq:LH}, we obtain
  \begin{equation}\label{eq:H2}
  \begin{split}
    0\geq &\ \rho\varphi\partial_1\psi+(k+1)\rho\varphi'\psi u_1+\varphi\psi u_1\frac{\rho_t}{u_t}-u_t\bigg(\rho\bigg(\varphi''-\frac{2\varphi'^2}{\varphi}\bigg)u_1S_k^{ij}u_iu_j \\
    &\ +\varphi u_1S_k^{ij}\rho_{ij}-2\varphi'u_1S_k^{ij}\rho_iu_j
    -\frac{2\varphi u_1}{\rho}S_k^{ij}\rho_i\rho_j\bigg).
    \end{split}
  \end{equation}
  It is easy to see that $\varphi'>0$, $\varphi''-\frac{2\varphi'^2}{\varphi}\geq\frac{1}{16}M^{-\frac{5}{2}}$,
  and
  $$\quad\, \, \, ~~~~S_k^{ij}u_iu_j=S_k^{11}u_1^2,\quad\,\, S_k^{ij}\rho_{ij}=\frac{-2}{r^2}\sum_{i=1}^nS_k^{ii},$$
  $$S_k^{ij}\rho_iu_j\leq\frac{2u_1}{r}\sum_{i=1}^nS_k^{ii},\quad S_k^{ij}\rho_i\rho_j\leq\frac{4}{r^2}\sum_{i=1}^nS_k^{ii}.$$
  Multiplying \eqref{eq:H2} by $M^{\frac{5}{2}}$, we thus get
  \begin{equation}\label{eq:H3}
    0\geq-32\rho M^{2}|\partial_1\psi|-\frac{32S_k M^2u_1}{r^2}-u_t\bigg(\rho S_k^{11}u_1^3-256\sum_{i=1}^nS_k^{ii}\bigg(\frac{M^2u_1}{r^2}+\frac{Mu_1^2}{r}+\frac{M^2u_1}{\rho r^2}\bigg)\bigg).
  \end{equation}

  To prove \eqref{eq:G}, we will assume $\rho u_1\geq\frac{10M}{r}$.
  Utilizing \eqref{eq:H1}, we get
  $$u_{11}\leq-\frac{\varphi'}{\varphi}u_1^2+\frac{2u_1}{\rho r}\leq-\frac{\varphi'}{2\varphi}u_1^2<0.$$
  With $u_{11}<0$, there holds (see \cite[(3.10)]{Chou-Wang 2003})
  \begin{equation}\label{eq:H5}
    S_k^{11}\geq S_{k-1}=\theta_1\sum_{i=1}^nS_k^{ii}.
  \end{equation}
  Here and in the following, we denote $\theta_j$ $(j\geq1)$ some positive constant depending only on $n$ and $k$. Also, the Maclaurin inequality gives
  \begin{equation}\label{eq:H5-1}
    S_k^{11}\geq S_{k-1}\geq\theta_2 S_k^{\frac{k-1}{k}}.
  \end{equation}
  From the proof of \cite[Theorem 3.2]{Chou-Wang 2003}, we see that
  \begin{equation}\label{eq:H6}
    \sum_{i=1}^nS_k^{ii}\geq\theta_3\frac{u_1^{2k-2}}{M^{k-1}}\geq\theta_4\frac{M^{k-1}}{r^{2k-2}}.
  \end{equation}
  Multiplying \eqref{eq:H3} by $\frac{\rho^2}{(-u_t)S_k^{11}}$, and putting \eqref{eq:u-t-bdd0} and \eqref{eq:H5}-\eqref{eq:H6} into \eqref{eq:H3}, we obtain
  \begin{equation}\label{eq:H4}
  \begin{split}
      \rho^3u_1^3\leq &\ \frac{32\rho^3M^{2}|\partial_1\psi|}{(-u_t)S_k^{11}}+\frac{32\rho M^2u_1S_k}{(-u_t)r^2S_k^{11}}
     +\frac{256\sum\limits_{i=1}^nS_k^{ii}}{S_k^{11}}\bigg(\frac{\rho^2M^2u_1}{r^2}+\frac{\rho^2Mu_1^2}{r}+\frac{\rho M^2u_1}{r^2}\bigg) \\
      \leq &\ C_1(1+\rho u_1)+C_2\bigg(\frac{\rho M^2u_1}{r^2}+\frac{\rho^2Mu_1^2}{r}\bigg) \\
     \leq &\ C_1+\frac{\rho^3u_1^3}{2}+C_2\frac{M^3}{r^3}, \\
   \end{split}
  \end{equation}
  where the last ``$\leq$'' comes from the Young's inequality. When $\psi\equiv\text{const}$, $\partial_1\psi$ vanishes in \eqref{eq:H4}, and so does $C_1$. This completes the proof of Theorem \ref{thm:gradient-est}.
\end{proof}

\section{Pogorelov estimates for $k$-convex-monotone solutions}\label{Sec3}

In this section, we will establish Pogorelov estimates for $k$-convex-monotone solutions to the Dirichlet boundary problem of equation \eqref{eq:P-k-HESSIAN}. This is a crucial ingredient to prove Theorem \ref{thm:main}, and is also of independent interest. The idea of the proof is adapted from that of \cite{Chou-Wang 2003}.

For $\lambda=(\lambda_1,\cdots,\lambda_n)$ and $1\leq l\leq n$, we define
\begin{equation*}
  \sigma_{l;i_1\cdots i_j}(\lambda)=
  \begin{cases}
    \sigma_l(\lambda)|_{\lambda_{i_1}=\cdots=\lambda_{i_j}=0}, & \mbox{if } i_r\neq i_s\ \mbox{for all }1\leq r<s\leq j, \\
    0, & \mbox{otherwise}.
  \end{cases}
\end{equation*}

\begin{proof}[Proof of Theorem \ref{thm:Pogorelov}]
  Consider the auxiliary function, for $(x,t)\in\overline{D}$ and $\xi\in\partial B_1$,
  $$\Phi(x,t;\xi)=(w(x,t)-u(x,t))^4\varphi\bigg(\frac{|D u(x,t)|^2}{2}\bigg)u_{\xi\xi}(x,t),$$
  where $\varphi(s)=\big(1-\frac{s}{M}\big)^{-\frac{1}{8}}$ and $M=2\sup_D|Du|^2$. Assume $\Phi$ attains its maximum at $(x_0,t_0;\xi_0)$. It suffices to prove \begin{equation}\label{eq:Phi0}
  \Phi(x_0,t_0;\xi_0)\leq C.
  \end{equation}
  We may assume $(x_0,t_0)\in\overline{D}\setminus\partial_pD$ and $u(x_0,t_0)<w(x_0,t_0)$, as otherwise \eqref{eq:Phi0} is obvious. By an orthogonal argument, we may also assume $\xi_0=(1,0,\cdots,0)$ and $D^2u(x_0,t_0)$ is diagonal with $u_{11}(x_0,t_0)\geq u_{22}(x_0,t_0)\geq\cdots\geq u_{nn}(x_0,t_0)$. It is enough to prove for $u_{11}(x_0,t_0)\geq1$.

  We set
  $$\Psi(x,t)=(w(x,t)-u(x,t))^4\varphi\bigg(\frac{|D u(x,t)|^2}{2}\bigg)u_{11}(x,t)\quad\text{for }(x,t)\in\overline{D}.$$
  Then, at $(x_0,t_0)$,
  \begin{equation}\label{eq:log H-1}
    0=(\log\Psi)_i=4\frac{w_i-u_i}{w-u}+\frac{\varphi_i}{\varphi}+\frac{u_{11i}}{u_{11}},
  \end{equation}
  \begin{equation}\label{eq:log H-2}
    0\leq(\log\Psi)_t=4\frac{w_t-u_t}{w-u}+\frac{\varphi_t}{\varphi}+\frac{u_{11t}}{u_{11}},
  \end{equation}
  \begin{equation}\label{eq:log H-3}
    0\geq(\log\Psi)_{ii}=4\left(\frac{w_{ii}-u_{ii}}{w-u}-\frac{(w_i-u_i)^2}{(w-u)^2}\right)+\frac{\varphi_{ii}}{\varphi}
    -\frac{\varphi_i^2}{\varphi^2}+\frac{u_{11ii}}{u_{11}}-\frac{u_{11i}^2}{u_{11}^2},
  \end{equation}
  for $i=1,2,\cdots,n$. Let $ F(D^2u)=S_k(D^2u)^{\frac{1}{k}}$ and $f=\psi^{\frac{1}{k}}$, then
  \begin{equation}\label{eq:P-Hessian1}
    (-u_t)^{\frac{1}{k}}F(D^2u)=f\quad\text{in }D.
  \end{equation}
  Let $L$ be the linearized operator of the above equation with respect to $u$ at $(x_0,t_0)$, defined by
  $$L=\frac{f(x_0,t_0,u(x_0,t_0))}{ku_t(x_0,t_0)}\frac{\partial}{\partial t}+(-u_t(x_0,t_0))^{\frac{1}{k}}F_{ij}(D^2u(x_0,t_0))D_{ij},$$
  where
  $$F_{ij}=F_{ij}(D^2u)=\frac{\partial F}{\partial u_{ij}}(D^2u).$$
  Since $u$ is $k$-convex-monotone, combining \eqref{eq:log H-2} and \eqref{eq:log H-3}, we have at $(x_0,t_0)$,
  \begin{equation}\label{eq:linear}
    \begin{split}
      0 \geq &\ L(\log\Psi) \\
      = &\ \frac{f}{ku_t}\left(4\frac{w_t-u_t}{w-u}+\frac{\varphi_t}{\varphi}+\frac{u_{11t}}{u_{11}}\right) \\
        & + (-u_t)^{\frac{1}{k}}F_{ii}\left(4\frac{w_{ii}-u_{ii}}{w-u}-4\frac{(w_i-u_i)^2}{(w-u)^2}+\frac{\varphi_{ii}}{\varphi}
         -\frac{\varphi_i^2}{\varphi^2}+\frac{u_{11ii}}{u_{11}}-\frac{u_{11i}^2}{u_{11}^2}\right) \\
      = &\ \frac{f}{ku_t}\bigg(4\frac{w_t-u_t}{w-u}+\frac{\varphi'}{\varphi}\sum_{j=1}^nu_ju_{jt}+\frac{u_{11t}}{u_{11}}\bigg) \\
        & +(-u_t)^{\frac{1}{k}}F_{ii}\bigg(4\frac{w_{ii}-u_{ii}}{w-u}-4\frac{(w_i-u_i)^2}{(w-u)^2}+\frac{\varphi''}{\varphi}u_i^2u_{ii}^2
        +\frac{\varphi'}{\varphi}\bigg(u_{ii}^2+\sum_{j=1}^nu_ju_{iij}\bigg) \\
        &-\frac{\varphi'^2}{\varphi^2}u_i^2u_{ii}^2+\frac{u_{11ii}}{u_{11}}-\frac{u_{11i}^2}{u_{11}^2}\bigg).
    \end{split}
  \end{equation}

  We begin to estimate each term on the right hand side of \eqref{eq:linear}. Differentiating \eqref{eq:P-Hessian1} with respect to $x_\gamma$ gives
  \begin{equation}\label{eq:D1}
    \frac{fu_{\gamma t}}{ku_t}+(-u_t)^{\frac{1}{k}}F_{ij}u_{ij\gamma}=\partial_\gamma f:=f_{x_\gamma}+f_zu_\gamma.
  \end{equation}
  Multiplying \eqref{eq:D1} by $(-u_t)^{-\frac{1}{k}}$, differentiating with respect to $x_\gamma$ again and multiplying $(-u_t)^{\frac{1}{k}}$, we get
  \begin{equation}\label{eq:D2}
  \begin{split}
    & -\left(1+\frac{1}{k}\right)\frac{fu_{\gamma t}^2}{ku_t^2}+\frac{2\partial_\gamma fu_{\gamma t}}{ku_t}+\frac{fu_{\gamma\gamma t}}{ku_t}+(-u_t)^{\frac{1}{k}}F_{ii}u_{ii\gamma\gamma}+(-u_t)^{\frac{1}{k}}F_{ij,rs}u_{ij\gamma}u_{rs\gamma} \\
    & =\partial_{\gamma\gamma}f:=f_{x_\gamma x_\gamma}+f_zu_{\gamma\gamma}+f_{zz}u_\gamma^2.
  \end{split}
  \end{equation}
  From now on, all formulas are assumed to hold at $(x_0,t_0)$. Since $w$ is $k$-convex-monotone, using the fact that $F$ is concave for $k$-convex functions and $F$ is homogeneous of degree $1$, we have $F_{ii}w_{ii}\geq0$ and
  \begin{equation*}
    F_{ii}(w_{ii}-u_{ii})\geq-F_{ii}u_{ii}=-F.
  \end{equation*}
  Then
  \begin{equation}\label{eq:w-est}
    \frac{f}{ku_t}\frac{w_t-u_t}{w-u}+(-u_t)^{\frac{1}{k}}F_{ii}\frac{w_{ii}-u_{ii}}{w-u}\geq-\bigg(1+\frac{1}{k}\bigg)\frac{f}{w-u}.
  \end{equation}
  In view of \eqref{eq:D1}, we have
  \begin{equation}\label{eq:linear-1}
     \frac{f}{ku_t}\sum_{j=1}^nu_ju_{jt}+(-u_t)^{\frac{1}{k}}F_{ii}\sum_{j=1}^nu_ju_{iij}=\sum_{j=1}^n\partial_\gamma fu_j\geq-C.
  \end{equation}
  Here and in the following, we denote $C$ some positive constant depending only on the desired quantities that may change line from line.
  From the calculation in \cite[Section 4]{Chou-Wang 2003}, it follows that
  \begin{equation}\label{eq:linear-3}
    F_{ij,rs}u_{ij\gamma}u_{rs\gamma}\leq-\frac{1}{k}\sum_{i,j=1}^n\sigma_k^{\frac{1}{k}-1}\sigma_{k-2;ij}u_{ij\gamma}^2.
  \end{equation}
  Setting $\gamma=1$ in \eqref{eq:D2}, \eqref{eq:linear-3} and Cauchy inequality yield
  \begin{equation}\label{eq:linear-2}
  \begin{split}
    &\ \frac{fu_{11t}}{ku_tu_{11}}+(-u_t)^{\frac{1}{k}}F_{ii}\frac{u_{11ii}}{u_{11}} \\
    \geq &\ \frac{\partial_{11}f}{u_{11}}+\left(1+\frac{1}{k}\right)\frac{fu_{1t}^2}{ku_t^2u_{11}}-\frac{2\partial_1fu_{1t}}{ku_tu_{11}}
     +(-u_t)^{\frac{1}{k}}\frac{1}{k}\sum_{i,j=1}^n\sigma_k^{\frac{1}{k}-1}\sigma_{k-2;ij}\frac{u_{1ij}^2}{u_{11}} \\
    \geq &\ \frac{\partial_{11}f}{u_{11}}-\frac{1}{k}\bigg(1+\frac{1}{k}\bigg)^{-1}\frac{(\partial_1f)^2}{fu_{11}}+(-u_t)^{\frac{1}{k}}\frac{1}{k}\sum_{i,j=1}^n\sigma_k^{\frac{1}{k}-1}
    \sigma_{k-2;ij}\frac{u_{1ij}^2}
    {u_{11}} \\
    \geq &\ -C+(-u_t)^{\frac{1}{k}}\frac{1}{k}\sum_{i,j=1}^n\sigma_k^{\frac{1}{k}-1}\sigma_{k-2;ij}\frac{u_{1ij}^2}{u_{11}} .
  \end{split}
  \end{equation}
  Substituting \eqref{eq:w-est}, \eqref{eq:linear-1}, \eqref{eq:linear-2} into \eqref{eq:linear}, we obtain
  \begin{equation}\label{eq:linear-4}
     \begin{split}
        0\geq & -4\bigg(1+\frac{1}{k}\bigg)\frac{f}{w-u}+(-u_t)^{\frac{1}{k}}\frac{1}{k}\sum_{i,j=1}^n\sigma_k^{\frac{1}{k}-1}\sigma_{k-2;ij}\frac{u_{1ij}^2}{u_{11}}-C \\
        & +(-u_t)^{\frac{1}{k}}F_{ii}\bigg(-4\frac{(w_i-u_i)^2}{(w-u)^2}+\frac{\varphi''}{\varphi}u_i^2u_{ii}^2+\frac{\varphi'}{\varphi}u_{ii}^2
        -\frac{\varphi'^2}{\varphi^2}u_i^2u_{ii}^2-\frac{u_{11i}^2}{u_{11}^2}\bigg).
     \end{split}
  \end{equation}

  Now we consider two cases as in \cite{Chou-Wang 2003}.

  \noindent\textbf{Case 1.} $u_{kk}\geq\varepsilon u_{11}$, where $\varepsilon>0$ is a small constant to be determined. Note that the second term of the right hand side of \eqref{eq:linear-4} is nonnegative. From \cite[Section 4]{Chou-Wang 2003}, it follows that
  \begin{equation*}\label{eq:linear-6}
    \sum_{i=1}^nF_{ii}u_{ii}^2\geq F_{kk}u_{kk}^2\geq\theta_1\sum_{i=1}^nF_{ii}u_{11}^2.
  \end{equation*}
  Here and in the following, we denote $\theta_j$ $(j\geq1)$ some positive constant depending only on the desired quantities.
  By \eqref{eq:log H-1}, we have
  \begin{equation*}\label{eq:linear-5}
    \frac{u_{11i}^2}{u_{11}^2}=\bigg(4\frac{w_i-u_i}{w-u}+\frac{\varphi_i}{\varphi}\bigg)^2\leq2\bigg(16\frac{(w_i-u_i)^2}{(w-u)^2}
    +\frac{\varphi'^2u_i^2u_{ii}^2}{\varphi^2}\bigg).
  \end{equation*}
  Note that $\frac{\varphi''}{\varphi}-3\frac{\varphi'^2}{\varphi^2}\geq0$. Plugging the estimates above into \eqref{eq:linear-4}, we have
  \begin{equation}\label{eq:linear-7}
  \begin{split}
    0 \geq & -4\bigg(1+\frac{1}{k}\bigg)\frac{f}{w-u}-C \\
    & +(-u_t)^{\frac{1}{k}}F_{ii}\bigg(-36\frac{(w_i-u_i)^2}{(w-u)^2}
    +\bigg(\frac{\varphi''}{\varphi}-3\frac{\varphi'^2}{\varphi^2}\bigg)u_i^2u_{ii}^2+\theta_1\frac{\varphi'}{\varphi}u_{11}^2\bigg) \\
    \geq & -4\bigg(1+\frac{1}{k}\bigg)\frac{f}{w-u}+\theta_2(-u_t)^{\frac{1}{k}}F_{ii}u_{11}^2-\frac{C(-u_t)^{\frac{1}{k}}}{(w-u)^2}F_{ii}-C.
    \end{split}
  \end{equation}
  Multiplying \eqref{eq:linear-7} by $\frac{(w-u)^8\varphi^2}{\theta_2(-u_t)^{\frac{1}{k}}F_{ii}}$, we obtain
  \begin{equation}\label{eq:linear-8}
    \begin{split}
    (w-u)^8\varphi^2u_{11}^2\leq &\ \frac{1}{\theta_2(-u_t)^{\frac{1}{k}}F_{ii}}\bigg(4\bigg(1+\frac{1}{k}\bigg)f(w-u)^7\varphi^2+C(w-u)^8\varphi^2\bigg) \\
    &\ +C(w-u)^6\varphi^2.
    \end{split}
  \end{equation}
  Note that (see \cite[(3.2)]{Chou-Wang 2003})
  \begin{equation}\label{eq:linear-9}
    \sum_{i=1}^nF_{ii}\geq F_{nn}=\frac{1}{k}\sigma_k^{\frac{1}{k}-1}\sigma_{k-1;n}\geq\theta_3\sigma_k^{\frac{1}{k}-1}u_{11}\cdots u_{k-1,k-1}\geq\theta_4\frac{f^{1-k}u_{11}^{k-1}}{(-u_t)^{\frac{1}{k}-1}}.
  \end{equation}
  By \eqref{eq:linear-8} and \eqref{eq:linear-9}, we get
  $$\Psi^2\leq \frac{Cf^{k-1}}{-u_tu_{11}^{k-1}}\bigg(4\bigg(1+\frac{1}{k}\bigg)f(w-u)^7\varphi^2+C(w-u)^8\varphi^2\bigg)+C(w-u)^6\varphi^2\leq C,$$
  which is the desired estimate.

  \noindent\textbf{Case 2.} $u_{kk}\leq\varepsilon u_{11}$. By \eqref{eq:log H-1}, we have
  \begin{equation}\label{eq:linear-10}
    \frac{u_{111}}{u_{11}}=-\bigg(\frac{\varphi_1}{\varphi}+4\frac{w_1-u_1}{w-u}\bigg)\quad\text{and}\quad\frac{w_i-u_i}{w-u}
    =-\frac{1}{4}\bigg(\frac{\varphi_i}{\varphi}+\frac{u_{11i}}{u_{11}}\bigg),
  \end{equation}
  for $i=2,\cdots,n$. Substituting \eqref{eq:linear-10} into \eqref{eq:linear-4}, we obtain
  \begin{equation}\label{eq:linear-11}
     \begin{split}
        0\geq & -4\bigg(1+\frac{1}{k}\bigg)\frac{f}{w-u}+(-u_t)^{\frac{1}{k}}\frac{1}{k}\sum_{i,j=1}^n\sigma_k^{\frac{1}{k}-1}\sigma_{k-2;ij}\frac{u_{1ij}^2}{u_{11}}-C \\
        & +(-u_t)^{\frac{1}{k}}F_{11}\bigg(-4\frac{(w_1-u_1)^2}{(w-u)^2}-\bigg(\frac{\varphi_1}{\varphi}+4\frac{w_1-u_1}{w-u}\bigg)^2\bigg) \\
        & +(-u_t)^{\frac{1}{k}}\sum_{i=2}^nF_{ii}\bigg(-\frac{1}{4}\bigg(\frac{\varphi_i}{\varphi}+\frac{u_{11i}}{u_{11}}\bigg)^2-\frac{u_{11i}^2}{u_{11}^2}\bigg)
        \\
        & +(-u_t)^{\frac{1}{k}}\sum_{i=1}^nF_{ii}\bigg(\frac{\varphi''}{\varphi}u_i^2u_{ii}^2+\frac{\varphi'}{\varphi}u_{ii}^2
        -\frac{\varphi'^2}{\varphi^2}u_i^2u_{ii}^2\bigg) \\
        \geq & \bigg[-4\bigg(1+\frac{1}{k}\bigg)\frac{f}{w-u}+(-u_t)^{\frac{1}{k}}\frac{1}{k}\sum_{i,j=1}^n\sigma_k^{\frac{1}{k}-1}\sigma_{k-2;ij}\frac{u_{1ij}^2}{u_{11}}
        -\frac{3}{2}(-u_t)^{\frac{1}{k}}\sum_{i=2}^nF_{ii}\frac{u_{11i}^2}{u_{11}^2}\bigg] \\
        & +\bigg[(-u_t)^{\frac{1}{k}}\sum_{i=1}^nF_{ii}\bigg(\bigg(\frac{\varphi''}{\varphi}-3\frac{\varphi'^2}{\varphi^2}\bigg)u_i^2u_{ii}^2+\frac{\varphi'}{\varphi}u_{ii}^2\bigg)
        -36(-u_t)^{\frac{1}{k}}F_{11}\frac{(w_1-u_1)^2}{(w-u)^2}\bigg] \\
        & -C \\
        =: &\ \mathrm{\MakeUppercase{\romannumeral1}}_1+\mathrm{\MakeUppercase{\romannumeral1}}_2-C.
  \end{split}
  \end{equation}

  We first estimate $\mathrm{\MakeUppercase{\romannumeral1}}_2$ from below. It is clear that
  \begin{equation}\label{eq:linear-12}
    \mathrm{\MakeUppercase{\romannumeral1}}_2\geq \theta_5(-u_t)^{\frac{1}{k}}F_{11}u_{11}^2-C(-u_t)^{\frac{1}{k}}\frac{F_{11}}{(w-u)^2}
    \geq\frac{\theta_5}{2}(-u_t)^{\frac{1}{k}}F_{11}u_{11}^2,
  \end{equation}
  provided $(w-u)^2u_{11}^2\geq\frac{2C}{\theta_5}$. We may assume that $\Psi(x_0,t_0)$ is so large that $(w-u)^2u_{11}^2\geq\frac{2C}{\theta_5}$, as otherwise $\Psi(x_0,t_0)\leq C$ holds obviously. From \cite[Section 4]{Chou-Wang 2003}, it follows that
  \begin{equation}\label{eq:linear-13}
    \frac{1}{k}\sum_{i,j=1}^n\sigma_k^{\frac{1}{k}-1}\sigma_{k-2;ij}\frac{u_{1ij}^2}{u_{11}}
        -\frac{3}{2}\sum_{i=2}^nF_{ii}\frac{u_{11i}^2}{u_{11}^2}\geq0,
  \end{equation}
  when $\varepsilon>0$ is sufficiently small. Putting \eqref{eq:linear-12} and \eqref{eq:linear-13} into \eqref{eq:linear-11}, we obtain
  \begin{equation}\label{eq:linear-14}
    0\geq-4\bigg(1+\frac{1}{k}\bigg)\frac{f}{w-u}+\frac{\theta_5}{2}(-u_t)^{\frac{1}{k}}F_{11}u_{11}^2-C.
  \end{equation}
  Multiplying \eqref{eq:linear-14} by $(w-u)^4\varphi$, we get
  \begin{equation*}
    0\geq-4\bigg(1+\frac{1}{k}\bigg)f(w-u)^3\varphi+\frac{\theta_5}{2}(-u_t)^{\frac{1}{k}}F_{11}(w-u)^4\varphi u_{11}^2-C(w-u)^4\varphi.
  \end{equation*}
  It follows from \cite[Lemma 3.1]{Chou-Wang 2003} that $\lambda_1\sigma_{k-1;1}(\lambda)\geq\theta_6\sigma_k(\lambda)$. This yields that
  $$F_{11}u_{11}^2=\frac{1}{k}\sigma_k^{\frac{1}{k}-1}\sigma_{k-1;1}u_{11}^2\geq\frac{\theta_6f}{k(-u_t)^{\frac{1}{k}}}u_{11}.$$
  We thus obtain the desired estimate
  $$\Psi\leq\frac{2k}{\theta_5\theta_6f}\bigg(4\bigg(1+\frac{1}{k}\bigg)f(w-u)^3\varphi+C(w-u)^4\varphi\bigg)\leq C.$$
  This completes the proof of Theorem \ref{thm:Pogorelov}.
\end{proof}

\section{Proof of Theorem \ref{thm:main}}\label{Sec4}

In this section, we will prove Theorem \ref{thm:main} by applying the interior estimates established in Section \ref{Sec2} and Section \ref{Sec3}.
%The proof is standard.

To begin with, we recall the Evans-Krylov type theorem below, which will be needed later. Given $0<\alpha<1$ and a domain $D\subset\mathbb{R}^n\times(-\infty,0]$, we denote
$$[u]_{C^\alpha(\overline{D})}=\sup_{(x,t),(y,s)\in D\atop (x,t)\neq(y,s)}\frac{|u(x,t)-u(y,s)|}{(|x-y|^2+|t-s|)^{\frac{\alpha}{2}}},$$
and
$$||u||_{C^{2+\alpha,1+\frac{\alpha}{2}}(\overline{D})}=\sum_{2i+j\leq2}||D^i_tD^j_xu||_{C^0(\overline{D})}+[D^2u]_{C^{\alpha}(\overline{D})}
+[u_t]_{C^\alpha(\overline{D})},$$
where $D^i_tD^j_xu(x,t)$ denotes the $j$-th order derivative with respect to $x$ and $i$-th order derivative with respect to $t$ of $u(x,t)$, $i$ and $j$ are two non-negative integers .
We denote $\mathbb{S}^{n\times n}$ the space of $n\times n$ real symmetric matrices.

\begin{lemma}[\cite{GH-1998}]\label{thm:Evans-Krylov}
  Let $D=B_1\times(-1,0]$ and $u\in C^{4,2}(\overline{D})$ be a solution to
  $$E(u_t,D^2u)=0\quad\text{in }D,$$
  where $E(q,N)$ is defined for all $(q,N)\in\mathbb{R}\times\mathbb{S}^{n\times n}$ with $E(\cdot,N)\in C^{1}(\mathbb{R})$ for each $N\in\mathbb{S}^{n\times n}$, and $E\in C^2(\mathbb{R}^n\times\mathcal{N})$, where $\mathcal{N}$ is a neighborhood of $\{D^2u(x,t):(x,t)\in D\}$. Suppose that
  \begin{itemize}
    \item[(1)] $E$ is uniformly parabolic, i.e. there exist positive constants $\Lambda_1$ and $\Lambda_2$ such that
  \begin{equation}\label{eq:EK-1}
  -\Lambda_2\leq E_q(q,N)\leq-\Lambda_1,
  \end{equation}
  \begin{equation}\label{eq:EK-2}
  \Lambda_1||N_2||\leq E(q,N_1+N_2)-E(q,N_1)\leq\Lambda_2||N_2||,
  \end{equation}
  for all $q\in\mathbb{R}$ and $N_1$, $N_2\in\mathbb{S}^{n\times n}$ with $N_2$ nonnegative definite.
    \item[(2)] $E$ is concave with respect to $N$.
  \end{itemize}
  If $||u||_{C^{2,1}(\overline{D})}\leq K$, then there exists positives constant $C$ depending only on $n$, $\Lambda_1$, $\Lambda_2$, $K$, $E(0,0)$ and $0<\alpha=\alpha(n,\Lambda_1,\Lambda_2)<1$ such that
  $$\|u\|_{C^{2+\alpha,1+\frac{\alpha}{2}}(\overline{D}_{\frac{1}{2}})}\leq C,$$
  where $D_{\frac{1}{2}}=B_{\frac{1}{2}}\times(-\frac{1}{2},0]$.
\end{lemma}

We set $E(q,N)=(-q)^{\frac{1}{k}}\sigma_k(\lambda(N))^{\frac{1}{k}}-1$ for $(q,N)\in[-m_2,-m_1]\times\mathcal{N}$, where
$$\mathcal{N}=\bigg\{N\in\mathbb{S}^{n\times n}:
\frac{1}{m_2}\leq\sigma_k(\lambda(N))\leq\frac{1}{m_1},||N||\leq C_0\bigg\},$$
in which $m_1$, $m_2$ and $C_0$ are given positive constant and $m_2\geq m_1$. It is easy to check that \eqref{eq:EK-1} holds in $[-m_2,-m_1]\times\mathcal{N}$. We can also find from \cite{CNS-1985} that \eqref{eq:EK-2} and $(2)$ in Theorem \ref{thm:Evans-Krylov} hold in $[-m_2,-m_1]\times\mathcal{N}$. Due to the smoothness of $E$, we are able to extend $E$ to $\mathbb{R}\times\mathbb{S}^{n\times n}$ such that $E$ meets the assumptions in Lemma \ref{thm:Evans-Krylov}.

Summing up, we have all ingredients to present the proof of Theorem \ref{thm:main}.

\begin{proof}[Proof of Theorem \ref{thm:main}]
  Fix $R_0=\sqrt{2B}$. For $R>R_0$, we define
  $$\Omega_R=\{(x,t)\in\mathbb{R}^n\times(-\infty,0]:u(Rx,R^2t)<R^2\}.$$
  Since $u$ is decreasing in $t$, we see that $\Omega_R(t)$ is nondecreasing in $t$. By virtue of \eqref{eq:quadratic}, we have
  \begin{equation*}\label{eq:Omega_R-bdd}
    B_{\frac{1}{\sqrt{2A_2}}}\subset\Omega_R(0)\subset B_{\frac{1}{\sqrt{A_1}}}.
  \end{equation*}
  It follows from \eqref{eq:u-t-bdd} that, for any $(x,t)\in\Omega_R$,
  \begin{equation*}
  R^2>u(Rx,R^2t)\geq u(Rx,0)-m_1R^2t\geq-m_1R^2t.
  \end{equation*}
  From this,
  \begin{equation*}\label{eq:Omega_R-bdd2}
    t\geq-\frac{1}{m_1},\quad\forall\ (x,t)\in\Omega_R.
  \end{equation*}
  By the above, we conclude that $\Omega_R\subset\subset B_{1+\frac{1}{\sqrt{A_1}}}\times(-1-\frac{1}{m_1},0]$.

  For $R>R_0$, we define
  $$v(x,t)=\frac{u(Rx,R^2t)-R^2}{R^2}.$$
  Then $v\in C^{4,2}(\mathbb{R}^n\times(-\infty,0])$ is $k$-convex-monotone. Moreover, we have
  \begin{equation}\label{eq:v-t-bdd}
  m_1\leq-v_t\leq m_2\quad\text{in }\mathbb{R}^n\times(-\infty,0],\
  \end{equation}
  \begin{equation}\label{eq:v-quadratic}
  A_1|x|^2-1\leq v(x,0)\leq A_2|x|^2-\frac{1}{2},\quad\forall\ x\in\mathbb{R}^n.
  \end{equation}
  It is clear that
  \begin{equation*}
    \begin{cases}
      -v_t\sigma_k(\lambda(D^2v))=1 & \mbox{in } \Omega_R, \\
      v=0 & \mbox{on }\partial_p\Omega_R.
    \end{cases}
  \end{equation*}
  By \eqref{eq:v-t-bdd} and \eqref{eq:v-quadratic}, we get
  \begin{equation}\label{v-bdd}
  -1\leq v(x,t)\leq v(x,0)-m_2t\leq A_2|x|^2-\frac{1}{2}-m_2t\leq C
  \end{equation}
  in $B_{1+\frac{1}{\sqrt{A_1}}}\times(-1-\frac{1}{m_1},0]$. Here and in the following, $C\geq1$ denotes some constant independent of $R$ that may change from line to line. In view of \eqref{v-bdd} and the problem $v$ solves, we apply the gradient estimate in Theorem \ref{thm:gradient-est} to $v$ and obtain
  \begin{equation*}\label{eq:Dv}
    |Dv|\leq C\quad\text{in }\Omega_R.
  \end{equation*}
  For $s<0$, let
  $$O_s=\{(x,t)\in\Omega_R:v(x,t)<s\}.$$
  Taking $s=-\frac{1}{4}$. We apply Pogorelov estimates in Theorem \ref{thm:Pogorelov} to $v$ and obtain
  $$\bigg(v+\frac{1}{4}\bigg)^4|D^2v|\leq C\quad\text{in }O_{-\frac{1}{4}}.$$
  This yields that
  $$|D^2v|\leq C\quad\text{in }O_{-\frac{1}{3}}.$$

  Set $Q=\left\{(x,t)\in\mathbb{R}^n\times(-\infty,0]:|x|<\frac{1}{\sqrt{8A_2}},t>-\frac{1}{48m_2}\right\}$.
  Again by \eqref{eq:v-t-bdd} and \eqref{eq:v-quadratic}, we get
  $$v(x,t)\leq v(x,0)-m_2t\leq A_2|x|^2-\frac{1}{2}-m_2t<-\frac{1}{3}\quad\text{in }Q.$$
  Hence, $Q\subset O_{-\frac{1}{3}}$. It follows from Lemma \ref{thm:Evans-Krylov} that, for some $0<\alpha<1$,
  $$||v||_{C^{2+\alpha,1+\frac{\alpha}{2}}(\overline{Q'})}\leq C,$$
  where $Q'=\left\{(x,t)\in\mathbb{R}^n\times(-\infty,0]:|x|<\frac{1}{\sqrt{10A_2}},\ t>-\frac{1}{50m_2}\right\}$.
  We thus get
  $$[D^2u]_{C^\alpha(\overline{\widetilde{Q}})}=R^{-\alpha}[D^2v]_{C^\alpha(\overline{Q})}\leq CR^{-\alpha},$$
  $$[u_t]_{C^{\alpha}(\overline{\widetilde{Q}})}=R^{-\alpha}[v_t]_{C^\alpha(\overline{Q})}\leq CR^{-\alpha},$$
  where $\widetilde{Q}=\left\{(x,t)\in\mathbb{R}^n\times(-\infty,0]:|x|<\frac{R}{\sqrt{10A_2}},\ t>-\frac{R^2}{50m_2}\right\}$. Since $R>R_0$ is arbitrary, by letting $R\rightarrow\infty$, we obtain
  $$[D^2u]_{C^\alpha(\mathbb{R}^n\times(-\infty,0])}=0,$$
  $$[u_t]_{C^\alpha(\mathbb{R}^n\times(-\infty,0])}=0.$$
  Therefore, $u$ has the desired form. This finishes the proof of Theorem \ref{thm:main}.
\end{proof}

\end{document}